\documentclass[12pt, a4paper]{amsart}
\usepackage[margin=1.0in]{geometry}
\usepackage{breqn}
\usepackage{amsmath}

\usepackage[utf8]{inputenc}
\usepackage{amssymb,amscd,amsthm, verbatim,amsmath,color,fancyhdr, mathrsfs}
\usepackage{graphicx}

\usepackage{tikz, float, tikzscale}

\usepackage{mathtools}

\usepackage[hidelinks]{hyperref}

\definecolor{mycolour1}{RGB}{153,153,153}

\newcommand{\interval}[5]{
  \fill [mycolour1] ({-.01*(2*(#2-3)+1)*2\textheight},{-1*#3-.1}) rectangle ({-.01*(2*(#1-2))*2\textheight},{-1*#3+.1});
  \ifnum#4=1
  \node at ({-.01*(2*(#1-2))*2\textheight},{-1*#3+.3}) {$\langle #1$};
  \fi
  \draw ({-.01*(2*(#1-2))*2\textheight},{-1*#3-.15})--({-.01*(2*(#1-2))*2\textheight},{-1*#3+.15});
  \ifnum#5=1
  \node at ({-.01*(2*(#2-3)+1)*2\textheight},{-1*#3+.3}) {$#2\rangle$};
  \fi
  \draw ({-.01*(2*(#2-3)+1)*2\textheight},{-1*#3-.15})--({-.01*(2*(#2-3)+1)*2\textheight},{-1*#3+.15});
}

\newtheorem{thm}{Theorem}
\newtheorem*{thm*}{Theorem}

\newtheorem{cor}[thm]{Corollary}

\DeclareMathOperator{\diam}{diam}

\begin{document}
	
\title{Lower bounds for the Hausdorff dimension of expressible sets}
	
\author{Maiken Gravgaard}
\address{Maiken Gravgaard, Department of
  Mathematics, University of York, Deramore Lane, York, YO10 5GH,
  United Kingdom.}
\email{maiken.gravgaard@york.ac.uk}
\author{Jaroslav Han\v{c}l}
\address{Jaroslav Han\v{c}l, Department of Mathematics, University of
  Ostrava, 30. dubna 22, 701 03 Ostrava 1, Czech Republic.}
\email{hancl@osu.cz}
\author{Simon Kristensen}
\address{Simon Kristensen, Department of Mathematics, Aarhus
  University, Ny Munkegade 118, 8000 Aarhus C, Denmark.}
\email{sik@math.au.dk}

\begin{abstract}
  We obtain positive lower bounds on the Hausdorff dimension of sets
  of real numbers given by expressions of the form $\sum_{n=1}^\infty
  \frac{1}{a_n b_n}$, where $b_n$ satisfies some growth condition and
  $a_n$ lies in some set, possibly depending on $n$. As a consequence
  of our results, some of the irrational numbers arising from
  Erd\H{o}s' celebrated construction from 1976 are not Liouville
  numbers.
\end{abstract}

\maketitle

\section{Introduction}

In his seminal paper from 1975 \cite{zbMATH03578925}, Erd\H{o}s
famously showed that if $(a_n)_{n=1}^\infty$ is an increasing sequence
of integers with $a_n \ge n^{1+\varepsilon}$ for some $\varepsilon >
0$ and with $\limsup_{n \rightarrow \infty} a_n^{1/2^n} = \infty$,
then the real number
\[
\sum_{n=1}^{\infty} \frac{1}{a_n}
\]
is irrational. This was later strengthened by Han\v{c}l
\cite{zbMATH02103621}, who proved that the condition $\limsup_{n
  \rightarrow \infty} a_n^{1/2^n} = \infty$ may be replaced with the
condition that the sequence $(a_n^{1/2^n})_{n=1}^\infty$ is divergent
with lower limit $>1$.
	
It was seen, for instance in \cite{zbMATH06790024}, that this
condition cannot be weakened further in terms of growth. Indeed, if
$a_n$ is given by the recurrence $a_{n+1} = a_n^2 - a_n +1$, then
\[
\sum_{n=1}^{\infty} \frac{1}{a_n} = \frac{1}{a_1 - 1}.
\]
On the other hand, it is similarly easy to show that the sequence with
terms $a_n^{1/2^n}$ is convergent, and indeed it was shown by Aho and
Sloane \cite{zbMATH03435566} that for $a_1 = 2$, the limit is about
$1.264$, which is certainly greater than $1$.
	
At a first glance, it seems reasonable to attempt to prove results of
this kind by approximating the sum by its truncations. These will of
course be rationals, and as the denominators increase very rapidly
indeed, in the absence of cancellation, one would expect them to
converge rapidly enough to provide a proof of the irrationality of the
sum. However, as the example above shows, cancellation does appear and
for some sequences in a fairly dramatic way at the limit of the
method.

It is of some interest to know whether the method -- when it works --
yields numbers which are not extremely well approximated by
rationals. By this we mean numbers which are not Liouville
numbers. We recall that a Liouville number is a real number $\ell$
such that for any $w > 1$, there is a rational $p/q \in \mathbb{Q}$
for which
\[
0 < \left\vert \ell - \frac{p}{q} \right\vert < \frac{1}{q^w}.
\]
The set consists only of transcendental numbers, and is a set of
Hausdorff dimension (see below for a definition) zero.
	
The questions of whether or not a sum of the above form is irrational
leads us to the notion of an expressible set. Given a sequence $A =
(a_n)_{n=1}^\infty$, we will say that the expressible set of the
sequence is
\[
\mathcal{K}(A)= \left\{ x=\sum_{n=1}^\infty \frac{1}{d_na_n} : d_n \in
\mathbb{N}\right\}.
\]
An alternative way of stating Erd\H{o}s' result \cite{zbMATH03578925}
is that if $\limsup_{n \rightarrow \infty} a_n^{1/2^n} = \infty$, then
$\mathcal{K}(A)$ contains only irrational numbers. In the same paper,
Erd\H{o}s proved that if the growth condition is strengthened to the
condition that $\limsup_{n \rightarrow \infty} a_n^{1/t^n} = \infty$
for any $t \in \mathbb{N}$, then every element in $\mathcal{K}(A)$ is
a Liouville number. A consequence of Theorem \ref{thm:double-exp2Nn}
below is that this result is sharp. By this we mean that if there is a
$t \in \mathbb{N}$ so that $\limsup_{n \rightarrow \infty} a_n^{1/t^n}
< \infty$, then $\mathcal{K}(A)$ contains non-Liouville numbers.
	
Another natural question is then what growth conditions of a sequence
ensure that the expressible set contains only transcendental
numbers. An answer to this was obtained by Han\v{c}l
\cite{zbMATH01154148} in greater generality. Note that there is
nothing in our set-up requiring the sequence $(a_n)_{n=1}^\infty$ to
consist of integers. Han\v{c}l proved that if $a_n = p_n/q_n$ for a
sequence of co-prime integers $p_n, q_n$ and $n \in \mathbb{N}$ with
the sequence of numerators $p_n \le 2^{(3 + \beta)^n}$ and
denominators $q_n \ge 2^{(3+\alpha)^n}$ for some choice of $\alpha >
\beta > 0$, then the expressible set $\mathcal{K}(A)$ consists
exclusively of transcendental numbers. For sequences $A$ of algebraic
numbers, see Han\v{c}l and Nair \cite{zbMATH06790024}, Andersen and
Kristensen \cite{zbMATH07126822} and Laursen \cite{zbMATH07725151} for
related results.
	
From the point of view of measure, it was shown by Han\v{c}l, Nair and
\v{S}ustek \cite{zbMATH05200846} that of $a_n = p_n/q_n$, $\limsup_{n
  \rightarrow \infty} q_n^{1/3^n} = \infty$, $q_n \ge
n^{1+\varepsilon}$ and $p_n \le 2^{\log_2^\alpha a_n}$ for some
$\alpha \in (0,1)$, then $\mathcal{K}(A)$ has Lebesgue measure
zero. Under the same conditions, but with the first replaced by
$\limsup_{n \rightarrow \infty} q_n^{1/(3+\delta)^n} = \infty$, it was
shown by Han\v{c}l, Nair, Novotn\'y and \v{S}ustek
\cite{zbMATH06090451} that the Hausdorff dimension of the set is at
most $2/(2+\delta)$.
	
Han\v{c}l and \v{S}ustek \cite{zbMATH06714945} proved that
$\mathcal{K}(A)$ has zero Lebesgue measure for the sequence with $a_n
= 2^{3^n}$. Furthermore, if $a_{n+1} \ge n 2^{2^{a_n}}$, Han\v{c}l and
\v{S}ustek \cite{zbMATH05227390} showed that $\mathcal{K}(A)$ has
Hausdorff dimension zero. This of course follows from Erd\H{o}s'
result \cite{zbMATH03578925} if we assume the $a_n$ to be integers,
since in this case we would see only Liouville numbers in
$\mathcal{K}(A)$, but they make no assumption on the integrality of
the sequence.
	
All the above results are concerned with sequences of doubly
exponential growth. As we will also be dealing with sequences of the
form $a_n = A^{n-1}$, we will briefly mention some results on the
expressible set of these sequences. In this case, $A$ can be assumed
to be any real number. It was shown by Han\v{c}l, Schintzel and
\v{S}ustek \cite{zbMATH05635538} that for $0 < A \le 1$,
$\mathcal{K}(A) = \mathbb{R}_+$. For $A > 1$, the series
$\sum_{n=1}^\infty \frac{1}{A^{n-1}} = \frac{1}{A-1}$, so there is no
hope for this to remain true, but for $1< A \le 3$, $\mathcal{K}(A) =
(0,\frac{1}{A-1}]$, the maximal possible. However, for $3 < A$, they
  were only able to prove that $\mathcal{K}(A) \supseteq (0,
  \frac{1}{(A-1)(\lceil A \rceil - 2)}]$.
    
 In the present paper, we will be considering expressible sets, but
 with additional restrictions. Concretely, we will consider various
 growth restrictions on the sequence $A = (a_n)_{n=1}^\infty$, but
 also restrictions on the `digits', $d_n$, which will be required to
 lie in some set $\mathbb{D}_n$, possibly depending on $n$. We will
 assume them to be natural numbers, but unless otherwise stated, we
 can assume the $a_n$ to be real numbers.
	
For a sequence $A = (a_n)_{n=1}^\infty$ and a sequence of sets of
natural numbers, $\left(\mathbb{D}_n\right)_{n=1}^{\infty}$, we define
\begin{equation}
  \label{eq:expressible}
  \mathcal{K}(A, (\mathbb{D}_n)_{n=1}^\infty) = \left\{
  x=\sum_{n=1}^\infty \frac{1}{d_na_n} : d_n \in C_n \right\}.
\end{equation}
Putting restrictions on the `digits' seems natural, but the only
previous results we are aware of are due to Han\v{c}l and \v{S}ustek
\cite{zbMATH05844058}, who considered expressible sets of integer
sequences, where the digits were assumed to be bounded. They
considered the set of numbers $x$ for which the digits $d_n$ can be
chosen boundedly with the bound depending on $x$, and for which the
base sequence $(a_n)_{n=1}^\infty$ is a sequence of positive integers,
satisfying that $a_n = 2^{2^{2^{[\log_2 n]}}} + b_n$, where $b_n =
O(2^{2^{n/2}})$. Criteria implying Lebesgue measure zero as well as
upper bounds on the Hausdorff dimension of the expressible set with
these restrictions were derived. In all cases, the sequence $A$ was
assumed to consist of integers.
 
\section{Background}
	
We will be concerned with the Hausdorff dimension of expressible
sets. Hausdorff dimension is defined as follows. For a set $E
\subseteq \mathbb{R}$ and real numbers $s \ge 0$, let
\[
\mathcal{H}^s(E) = \lim_{\delta \rightarrow 0} \inf\left\{ \sum_{U \in
  \mathcal{C}_\delta} \diam(U)^s : \mathcal{C}_\delta \text{ is a
  cover of $E$ with sets of diameter $\le \delta$}\right\}.
\]
The Hausdorff dimension of $E$ is 
\[
\dim_{\mathcal{H}}(E) = \inf\{s \ge 0: \mathcal{H}^s(E)=0\}.
\]
We refer to Falconer's book \cite{zbMATH01994007} for
details.

A class of examples of sets in $\mathbb{R}$ for which the Hausdorff
dimension may be estimated are (generalised) Cantor sets. One
constructs such a set by starting with the closed unit interval at
level $0$. This is split into $m_1$ closed intervals, which are kept
and where the distance between them is at most $\varepsilon_1$. The
procedure is repeated, so that each of the level $1$ intervals is
split into $m_2$ closed intervals, and so that the gaps between the
remaining intervals is at least $\varepsilon_2$. This procedure is
repeated \emph{ad infinitum}, and the resulting set denoted by
$\mathcal{C}$, see Figure \ref{fig: dim}.
	
\begin{figure}[h] 
  \centering
  \scriptsize
  \begin{tikzpicture} 
    \useasboundingbox (0,-0.45) rectangle (15,0.45);
    \fill [mycolour1] (0,-0.1) rectangle (15,0.1);
    \draw (0,0)--(15,0);
    \node at (0,.3) {Level 0:};
    \draw (0,-0.15)--+(0,0.3); 
    \draw (15,-0.15)--+(0,0.3);
    \draw (15,-0.15)--+(0,0.3);
  \end{tikzpicture}
  \begin{tikzpicture} 
    \useasboundingbox (0,-0.45) rectangle (15,0.45);
    \fill [mycolour1] (0,-0.1) rectangle (3,0.1);
    \fill [mycolour1] (5,-0.1) rectangle (7,0.1);
    \fill [mycolour1] (11,-0.1) rectangle (15,0.1);
    \node at (0,.3) {Level 1:};
    \node at (9,.15) {$\dots$};
    \draw (0,0)--(15,0);
    \draw (15,-0.15)--+(0,0.3);
    \draw (0,-0.15)--+(0,0.3);
    \draw (3,-0.15)--+(0,0.3);
    \draw (5,-0.15)--+(0,0.3);
    \draw (7,-0.15)--+(0,0.3);
    \draw (11,-0.15)--+(0,0.3);
    \draw (0,-0.15)--+(0,0.3); 
  \end{tikzpicture}
  \begin{tikzpicture} 
    \useasboundingbox (0,-0.45) rectangle (15,0.45);
    \fill [mycolour1] (0,-0.1) rectangle (0.5,0.1);
    \fill [mycolour1] (2,-0.1) rectangle (3,0.1);
    \fill [mycolour1] (5,-0.1) rectangle (5.33,0.1);
    \fill [mycolour1] (6.33,-0.1) rectangle (7,0.1);
    \fill [mycolour1] (11,-0.1) rectangle (11.66,0.1);
    \fill [mycolour1] (13.66,-0.1) rectangle (15,0.1);
    \node at (0,.3) {Level 2:};
    \node at (1.25,.15) {$\dots$};
    \node at (5.8,.15) {$\dots$};
    \node at (12.66,.15) {$\dots$};
    \draw (0,0)--(15,0);
    \draw (15,-0.15)--+(0,0.3);
    \draw (0,-0.15)--+(0,0.3);
    \draw (0.5,-0.15)--+(0,0.3);
    \draw (2,-0.15)--+(0,0.3);
    \draw (3,-0.15)--+(0,0.3);
    \draw (5,-0.15)--+(0,0.3);
    \draw (5.33,-0.15)--+(0,0.3);
    \draw (6.33,-0.15)--+(0,0.3);
    \draw (7,-0.15)--+(0,0.3);
    \draw (11.,-0.15)--+(0,0.3);
    \draw (11.66,-0.15)--+(0,0.3);
    \draw (13.66,-0.15)--+(0,0.3);
    \draw (0,-0.15)--+(0,0.3); 
  \end{tikzpicture}
  \begin{tikzpicture} 
    \useasboundingbox (0,-0.45) rectangle (15,0.45);
    \node at (7.5, 0.45) {$\vdots$};
  \end{tikzpicture}
  \caption{Generalised Cantor set with $m_n$ sub-intervals in layer
    $n$ and gaps of size $\geq \varepsilon_n$}
  \label{fig: dim}
\end{figure}
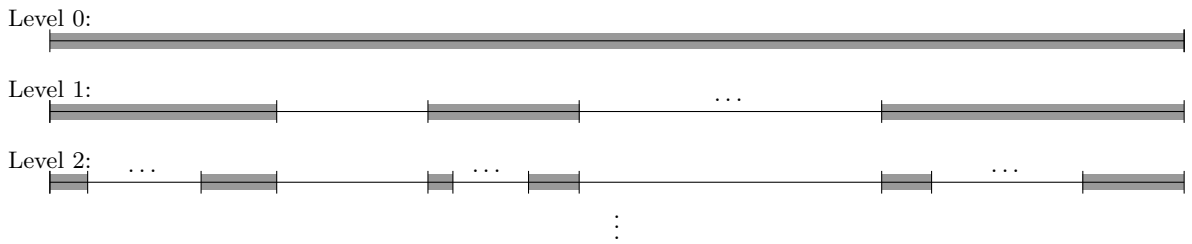
	
The following is stated as Example 4.6 in \cite{zbMATH01994007}. We
state it here as a theorem.
	
\begin{thm}
  \label{BF}
  Let $\mathcal{C}$ be a generalised Cantor set (see Figure \ref{fig:
    dim}). Assume that there exists a positive integer $n_0$ such
  that, for any $n\geq n_0$, each interval in level $n$ splits into
  $m_n$ subintervals and each gap in level $n$ is $\geq
  \varepsilon_n$. Then
  \begin{align*}
    \dim_{\mathcal{H}}\left(\mathcal{C}\right) \geq
    \limsup_{n\to\infty}\frac{\log(m_1\dots
      m_n)}{-\log(m_{n+1}\varepsilon_{n+1})}.
  \end{align*}
\end{thm}

\section{Theorems and proofs}
	
We will first describe the overall strategy of our proofs. It is
helpful to think of the series as a numeration system, so that the
sequence $A$ is thought of as the `base' and the varying elements from
the $\mathbb{D}_n$ are thought of as the `digits'.
	
For an increasing non-negative sequence $A= (a_n)_{n\in\mathbb{N}}$
and a finite set of `digits'
\[
\mathbb{D}_n = \{1=d_{n,1} < \dots < d_{n,
  m_n}\}\subset\mathbb{N}
\]
for each $n$, consider the set
\[
\mathcal{K}(A, (\mathbb{D}_n)_{n=1}^\infty) =
\left\{\sum_{n=1}^{\infty} \frac{1}{a_n d_n} : d_n\in\mathbb{D}_n
\text{ for each } n\in\mathbb{N}\right\}.
\]
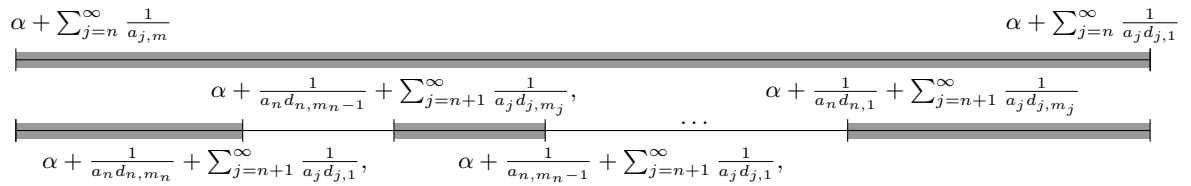
\begin{figure}[h] 
  \centering
  \scriptsize
  \begin{tikzpicture} 
    \useasboundingbox (0,-0.45) rectangle (15,0.45);
    \fill [mycolour1] (0,-0.1) rectangle (15,0.1);
    \draw (0,0)--(15,0);
    \node at (1,.45) {$\alpha  + \sum_{j=n}^{\infty} \frac{1}{a_{j, m}}$};
    \node at (14.25,.45) {$\alpha  + \sum_{j=n}^{\infty} \frac1{a_j d_{j, 1}}$};
    \draw (0,-0.15)--+(0,0.3); 
    \draw (15,-0.15)--+(0,0.3);
    \draw (15,-0.15)--+(0,0.3);
  \end{tikzpicture}
  \begin{tikzpicture} 
    \useasboundingbox (0,-0.45) rectangle (15,0.45);
    \fill [mycolour1] (0,-0.1) rectangle (3,0.1);
    \fill [mycolour1] (5,-0.1) rectangle (7,0.1);
    \fill [mycolour1] (11,-0.1) rectangle (15,0.1);
    \node at (2.5,-.45) {$\alpha  + \frac{1}{a_n d_{n, m_n}} +
      \sum_{j=n+1}^{\infty} \frac1{a_j d_{j, 1}},$};  
    \node at (5,.45) {$\alpha + \frac{1}{ a_n d_{n, m_n-1}} +
      \sum_{j=n+1}^{\infty} \frac1{a_j d_{j, m_j}},$}; 
    \node at (8,-.45) {$\alpha + \frac{1}{a_{n, m_n-1}} +
      \sum_{j=n+1}^{\infty} \frac1{a_j d_{j, 1}}$,}; 
    \node at (9,.1) {$\dots$};
    \node at (12,.45) {$\alpha + \frac{1}{a_n d_{n, 1}} +
      \sum_{j=n+1}^{\infty} \frac1{a_j d_{j, m_j}}$}; 
    \draw (0,0)--(15,0);
    \draw (15,-0.15)--+(0,0.3);
    \draw (0,-0.15)--+(0,0.3);
    \draw (3,-0.15)--+(0,0.3);
    \draw (5,-0.15)--+(0,0.3);
    \draw (7,-0.15)--+(0,0.3);
    \draw (11,-0.15)--+(0,0.3);
    \draw (0,-0.15)--+(0,0.3); 
  \end{tikzpicture}
  \caption{A splitting of an interval in level $n-1$ into subintervals
    in level $n$ of the Cantor set $\mathcal{K}$. Here $\alpha$ is of
    the form $\alpha=\frac1{a_1 d_{1,k_1}}+ \dots + \frac1{a_{n-1}
      d_{1,k_{n-1}}}$, $1\leq k_i\leq m_i$ for each $1\leq i\leq
    n-1$, so that the expressions in the figure are indeed the end
    point of the intervals.}
  \label{fig: cantor}
\end{figure}

We wish to apply Theorem \ref{BF}. As such, it is important for us to
get estimates on the involved quantities, $m_n$ and
$\varepsilon_n$. In the proofs, $m_n$ will be either equal to or one
less than the number of digits $\#\mathbb{D}_n$, and so the quantity
will be part of the assumptions. This leaves us with the gap sizes,
which must be lower bounded by some $\varepsilon_n$. The main
difficulty in the proofs is to ensure that the conditions stated
ensure that gaps do indeed occur, and that their lengths may be lower
bounded by some function.

We now state and prove three theorems, where the sequence $A$ is
allowed different growth rates. In the first theorem, we consider the
case when the growth is geometric, i.e. if $a_n = b^n$ for some real
number $b > 1$. In this case, the result of Han\v{c}l, Schintzel and
\v{S}ustek \cite{zbMATH05635538} implies that the full expressible set
$\mathcal{K}(A)$ is of maximal Hausdorff dimension and in fact
contains an interval. However, restricting the digits to lie in some
small set will remove this property. We obtain the following.
	
\begin{thm}
  \label{thm:exp}
  Assume $a_n = b^n$ for some $b>4$ and $\mathbb{D}_n = \{1,2,\dots,
  K\}$ for every $n$, where $K^2< b$.  Then
  \[
  \dim_{\mathcal{H}}(\mathcal{K}((a_n)_{n=1}^{\infty},
  (\mathbb{D}_n)_{n=1}^\infty)) \geq \frac{\log(K)}{\log(b)}.
  \] 
\end{thm}

Note that the theorem remains valid, but trivial, for $1 < b \le
4$. Indeed, the condition that $K^2 < b$ will in this case imply that
$K=1$, so that the lower bound obtained on the dimension is equal to
zero.

\begin{proof}
  In this case $m_n=K$, but we must make sure that all of the expected
  gaps occur. We show this by giving a positive lower bound for the
  gaps. Let $n\in\mathbb{N}$ and $1\leq k \leq K-1$
  \begin{align*}
    \text{gap in level n}&=\left( \alpha + \frac{1}{b^n k} +
    \sum_{j=n+1}^{\infty} \frac{1}{K b^j}\right) - \left( \alpha +
    \frac{1}{b^n (k+1)} + \sum_{j=n+1}^{\infty} \frac{1}{ b^j}\right)
    \\
    &= \frac{1}{b^n}\left( \frac{1}{k(k+1)}\right) +
    \left(\frac{1}{K}-1\right)\frac{1}{b^n}\frac{1}{b-1} \\
    &\geq
    \frac{1}{b^n}\frac{1}{K}\left(\frac{1}{K+1}+(1-K)\frac{1}{b-1}\right) 
    =: \varepsilon_n > 0 \text{ since } K^2<b,
  \end{align*}
  so we have a valid Cantor set (meaning that there are in fact gaps
  everywhere we expect there to be) and expressions for $m_n$ and
  $\varepsilon_n$.  Now we plug it into Theorem \ref{BF}
  \begin{align*}
    \dim_{\mathcal{H}}(\mathcal{K}((a_n)_{n=1}^{\infty},
    (\mathbb{D}_n)_{n = 1}^\infty) &\geq \limsup_{n\to\infty}
    \frac{\log(K^n)}{-\log\left(K
      \frac{1}{b^{n+1}}\frac{1}{K}\left(\frac{1}{K+1}+(1-K)\frac{1}{b-1}
      \right)\right)} \\
    &= \limsup_{n\to\infty} \frac{n\log(K)}{(n+1)\log(b) -
      \log\left(\left(\frac{1}{K + 1}+(1 - K)\frac{1}{b -
        1}\right)\right)}
    \\ &= \frac{\log(K)}{\log(b)}.
  \end{align*}			
\end{proof}

In our second result, we consider sequences of doubly exponential
growth.
	
\begin{thm}\label{thm:double-exp2Nn}
  Let $N\in\mathbb{N}_{\geq 2}$ and $0<s\leq r < \frac{N-1}{2}$. Take
  functions $f, g, h: \mathbb{N} \rightarrow \mathbb{N}$ satisfying
  $1\leq f(n), g(n)^{\frac1s}, h(n)^{\frac1r}\leq 2^{(N-1)^n}$ and
  $2^{sN^n} g(n)\leq 2^{rN^n}h(n)$. Define for each $n\in\mathbb{N}$,
  \begin{equation*}
    a_n= 2^{N^n} f(n),
  \end{equation*}
  \begin{equation*}
    m_n = \lfloor 2^{sN^n}g(n) \rfloor,
  \end{equation*}
  and  
  \begin{equation*}
    \mathbb{D}_n = \{ 1=d_{n, 1}< d_{n,2}<\dots , d_{n, m_n}\leq
    2^{rN^n}h(n)\} \subset \mathbb{N}. 
  \end{equation*}
  Then
  \begin{equation*}
    \dim_{\mathcal{H}}(\mathcal{K}((a_n)_{n\in\mathbb{N}},
    \left(\mathbb{D}_n\right)_{n=1}^{\infty}))\geq
    \frac{s}{(N-1)(N-s)}.
  \end{equation*}
\end{thm}
\begin{proof}
  First we prove that there exists an $\varepsilon>0$ and an
  $n_0\in\mathbb{N}$ such that for all $n\geq n_0$ we have
  \begin{equation}\label{minbound2}
    \min_{x\neq y\in\mathbb{D}_n} \left\lvert \frac{1}{x a_n}-
    \frac{1}{y a_n} \right\rvert \geq
    (1+\varepsilon)\sum_{k=n+1}^{\infty} \frac{1}{a_k}. 
  \end{equation}
  Consider first the left hand side of (\ref{minbound2}):
  \begin{align*}
    \min_{x\neq y\in\mathbb{D}_n} \left\lvert \frac{1}{x a_n}-
    \frac{1}{y a_n} \right\rvert  &\geq
    \frac{1}{a_n}\left(\frac{1}{d_{n, m_n}-1} -
    \frac{1}{d_{n,m_n}}\right) \\ 
    &\geq \frac{1}{2^{N^{n}}f(n)}\frac{1}{2^{2rN^n}h(n)^2} \\
    &\geq \frac{1}{2^{(1+2r)(N^n+(N-1)^n)}}
  \end{align*}
  As for the right hand side of (\ref{minbound2}), we have
  \begin{align*}
    (1+\varepsilon)\sum_{k=n+1}^{\infty} \frac{1}{a_k} &\leq
    (1+\varepsilon)\frac{2}{2^{N^{n+1}}f(n)}\\
    &\leq (1+\varepsilon) \frac{2}{2^{N^{n+1}}}, 
  \end{align*}
  so we just need to show, that there is an $\varepsilon>0$ and an
  $n_0\in\mathbb{N}$ such that for all $n\geq n_0$ we have
  \begin{equation*}
    \frac{1}{2^{(1+2r)(N^n+(N-1)^n)}}\geq
    (1+\varepsilon)\frac{2}{2^{N^{n+1}}},
  \end{equation*}
  i.e.
  \begin{equation*}
    \varepsilon \leq \frac{2^{(N-2r-1)N^n-(1+2r)(N-1)^n}-2}{2}.
  \end{equation*}
  In the above expression, the right hand side goes to $\infty$ as
  $n\to\infty$ (since $N-2r-1>0$), so for any $\varepsilon>0$, there
  is an $n_0(\varepsilon)\in\mathbb{N}$ such that this holds, so the
  claim in true.
		
  Now to look at the gaps in the Cantor set. Let $n\geq n_0$. For
  $j\in\{2,\dots, m_n\}$ the $j$'th gap in layer $n$ is:
  \begin{align*}
    \text{gap}_n^{j} &= \frac{1}{d_{n,j-1}a_n} + \sum_{k=n+1}^{\infty}
    \frac{1}{d_{k, m_k}a_k} - \frac{1}{d_{n,j}a_n} -
    \sum_{k=n+1}^{\infty} \frac{1}{a_k} \\ 
    &\overset{(\ref{minbound2})}{\geq}
    (1+\varepsilon)\sum_{k=n+1}^{\infty} \frac{1}{a_k} -
    \sum_{k=n+1}^{\infty} \frac{1}{a_k} + \sum_{k=n+1}^{\infty}
    \frac{1}{d_{k,m_k} a_k} \\ 
    &\geq \varepsilon\sum_{k=n+1}^{\infty} \frac{1}{a_k} \geq
    \frac{\varepsilon}{2^{N^{n+1}}f(n+1)} =: \varepsilon_n > 0. 
  \end{align*}
  So after level $n_0$ all of the expected gaps do indeed occur, and
  we have a lower bound for the size of the gaps in layer $n$, namely
  $\varepsilon_n = \frac{\varepsilon}{2^{N^{n+1}}f(n+1)}$.
		
  For each $n\geq n_0$ we have, since $\lfloor 2^{sN^k} g(n)\rfloor
  \geq \frac12 2^{sN^k}g(n)\geq \frac12 2^{sN^k}$, 
  \begin{align}
    m_1\cdot \dots \cdot m_{n-1} &= \lfloor 2^{N^1}g(1)\rfloor\cdot
    \dots \cdot \lfloor 2^{N^{n-1}} g(n-1)\rfloor \nonumber\\ 
    &\geq \left(\frac{1}{2}\right)^{n-1} 2^{s(N^1+\dots + N^{n-1})}
    \nonumber\\ 
    &= 2^{s\frac{N^n-N}{N-1}-(n-1)} \label{m1...mn2}.
  \end{align}
  And for each $n\geq n_0$ we have
  \begin{align}
    \varepsilon_n m_n &= \frac{\varepsilon}{2^{N^{n+1}}f(n+1)}
    \lfloor 2^{sN^n}g(n)\rfloor \nonumber \\
    & \geq \frac{\varepsilon}{2}\cdot\frac{1}{2^{(N-s)N^n}}
    \frac{g(n)}{f(n+1)} \nonumber \\
    &\geq \frac{\varepsilon}{2} \cdot
    \frac{1}{2^{(N-s)N^n+(N-1)^{n+1}}}\label{epsilonnmn2},  
  \end{align}
  so
  \begin{equation*}
    -\log(\varepsilon_n m_n) \leq -\log(\frac{\varepsilon}{2}) +
    \log(2)\left[(N-s)N^n + (N-1)^{n+1}\right]. 
  \end{equation*}
  To find a lower bound for the Hausdorff dimension of
  $\mathcal{K}((a_n)_{n\in\mathbb{N}},
  \left(\mathbb{D}_n\right)_{n=1}^{\infty})$ we apply Theorem \ref{BF}
  as well as (\ref{m1...mn2}) and (\ref{epsilonnmn2}):
  \begin{align*}
    \dim_{\mathcal{H}}\left( \mathcal{K}((a_n)_{n\in\mathbb{N}},
    \left(\mathbb{D}_n\right)_{n=1}^{\infty}) \right) &\geq
    \limsup_{n\to\infty} \frac{\log(m_1\cdot \dots \cdot
      m_{n-1})}{-\log(\varepsilon_n m_n)} \\
    &\geq \limsup_{n\to\infty} \frac{\log(2)\left(s\frac{N^n - N}{N -
        1} - (n - 1)\right)}{-\log(\frac{\varepsilon}{2}) +
      \log(2)\left[(N-s)N^n+(N-1)^{n+1}\right]} \\
    &= \frac{s}{(N-1)(N-s)}.
  \end{align*}
\end{proof}

Our third theorem again concerns doubly exponential growth, but this
time slower than that of Theorem \ref{thm:double-exp2Nn}.

\begin{thm}\label{thm:double-expN2n}
  Let $N\in\mathbb{N}_{\geq 2}$ and $0<s\leq r<\frac12$. Let $0<\eta <
  1$. Take functions $f,g, h: \mathbb{N} \rightarrow \mathbb{N}$
  satisfying $1\leq f(n), g(n)^{\frac1s}, h(n)^{\frac1r}\leq
  N^{(2-\eta)^n}$ and $N^{s2^n} g(n)\leq N^{r2^n}h(n)$.  Define for
  each $n\in\mathbb{N}$,
  \begin{equation*}
    a_n= N^{2^n} f(n),
  \end{equation*}
  \begin{equation*}
    m_n = \lfloor N^{s2^n}g(n) \rfloor,
  \end{equation*}
  and  
  \begin{equation*}
    \mathbb{D}_n = \{ 1=d_{n, 1}< d_{n,2}<\dots , d_{n, m_n}\leq
    N^{r2^n}h(n)\} \subset \mathbb{N}. 
  \end{equation*}
  Then
  \begin{equation*}
    \dim_{\mathcal{H}}\left(\mathcal{K}((a_n)_{n=1}^{\infty},
    \left(\mathbb{D}_n\right)_{n=1}^{\infty})\right) \geq
    \frac{s}{2-s}.
  \end{equation*}
\end{thm}

\begin{proof}
  Again we have for some $\varepsilon>0$ and some $n_0\in\mathbb{N}$,
  that for any $n\geq n_0$:
  \begin{equation}
    \min_{x\neq y\in\mathbb{D}_n} \left\lvert\frac{1}{x a_n} -
    \frac{1}{y a_n} \right\rvert \geq
    (1+\varepsilon)\sum_{k=n+1}^{\infty} \frac{1}{a_k}. 
  \end{equation}
  Indeed
  \begin{equation*}
    \min_{x\neq y\in\mathbb{D}_n} \left\lvert\frac{1}{x a_n} -
    \frac{1}{y a_n} \right\rvert \geq \frac{1}{d_{n,m_n}^2 a_n} \geq
    \frac{1}{N^{(1+2r)2^n}N^{(1+2r)(2-\eta)^{n}}} 
  \end{equation*}
  and 
  \begin{align*}
    (1+\varepsilon)\sum_{k=n+1}^{\infty}\frac{1}{a_k} \leq
    (1+\varepsilon)\frac{2}{N^{2^{n+1}}} 
  \end{align*}
  and since $r<\frac12$, for any $\varepsilon>0$ there is an
  $n_0(\varepsilon)\in\mathbb{N}$ such that for all $n\geq
  n_0(\varepsilon)$ 
  \begin{align*}
    \varepsilon > \frac{N^{(1-2r)2^n-(1+2r)(2-\eta)^n}-2}{2}.
  \end{align*}
  Exactly as in the proof of Theorem \ref{thm:double-exp2Nn}, this
  implies that the gaps in layer $n$ of the Cantor set are $\geq
  \frac{\varepsilon}{N^{2^{n+1}}}$.  Furthermore we have
  \begin{align}
    m_1\cdot\dots\cdot m_{n-1} &= \lfloor
    N^{s2^1}g(1)\rfloor\cdot\dots\cdot\lfloor N^{s2^{n-1}}g(n-1)\rfloor
    \nonumber \\ 
    &\geq \left(\frac{1}{2}\right)^{n-1}N^{s(2^n-2)} \label{m1...mn3}
  \end{align}
  and
  \begin{align}
    \varepsilon_n m_n &= \frac{\varepsilon}{N^{2^{n+1}}} \lfloor
    N^{s2^n} g(n) \rfloor \nonumber \\
    &\geq \frac{\varepsilon}{2}\frac{1}{
      N^{(2-s)2^n}}. \label{epsilonnmn3} 
  \end{align}
  We find our bound for the Hausdorff dimension of
  $\mathcal{K}((a_n)_{n=1}^{\infty},
  \left(\mathbb{D}_n\right)_{n=1}^{\infty})$ by applying Theorem
  \ref{BF} as well as (\ref{m1...mn3}) and (\ref{epsilonnmn3}):
  \begin{equation*}
    \dim_{\mathcal{H}}\left(\mathcal{K}((a_n)_{n=1}^{\infty},
    \left(\mathbb{D}_n\right)_{n=1}^{\infty}) \right) \geq
    \limsup_{n\to\infty} \frac{\log(N)s2^n-\log(N)s2-\log(2)(n -
      1)}{-\log\left(\frac{\varepsilon}{2}\right) + (2-s)2^n\log(N)} =
    \frac{s}{2-s}.
  \end{equation*}
\end{proof}

In the case $N=2$ Theorem \ref{thm:double-exp2Nn} and Theorem
\ref{thm:double-expN2n} agree. It is worth noting that the bound
obtained in Theorem \ref{thm:double-expN2n} is independent of $N$,
whereas the bound in Theorem \ref{thm:double-exp2Nn} tends to zero as
$N$ increases. This is natural, as the result of Erd\H{o}s
\cite{zbMATH03578925} mentioned in the introduction says that if the
$a_n$ are integers, then the limiting set would contain only Liouville
numbers, which is a set of Hausdorff dimension $0$.

We note however that in the above results, the sequence $A =
(a_n)_{n=1}^\infty$ is not assumed to be a sequence of integers. Any
sequence of real numbers satisfying the required growth conditions
will suffice.
	
Assuming that the sequence consists of rational numbers or of
algebraic numbers brings us into the realms of the works of Han\v{c}l
\cite{zbMATH01154148}, Han\v{c}l and Nair \cite{zbMATH06790024},
Andersen and Kristensen \cite{zbMATH07126822} and Laursen
\cite{zbMATH07725151}. The following result is mostly relevant for
these settings. For sequences of integers, it shows in particular that
Erd\H{o}s' criterion on when an expressible set contains only
Liouville numbers is sharp.
	
\begin{cor}
  \label{cor:classification}
  Suppose that the $a_n$ are all integers with $\limsup_{n \rightarrow
    \infty} a_n^{1/t^n} < \infty$ for some $t\in\mathbb{N}$. Then
  $\mathcal{K}(A)$ contains numbers which are not Liouville numbers.
\end{cor}
	
\begin{proof}
  The growth condition on $a_n$ implies that $a_n \le 2^{N^n}$ for
  some $N \in \mathbb{N}_{\ge 2}$, which is the same as that in
  Theorem \ref{thm:double-exp2Nn}. Since $\mathbb{D}_n \subseteq
  \mathbb{N}$ for any choice of digit sets, $\mathcal{K}(A,
  (\mathbb{D}_n)_{n=1}^\infty) \subseteq \mathcal{K}(A)$, so choosing
  digits so that the Hausdorff dimension of the former is positive
  immediately implies the positivity of the Hausdorff dimension of
  $\mathcal{K}(A)$. Since the Hausdorff dimension of the set of
  Liouville numbers is equal to $0$, there must be other types of
  numbers in $\mathcal{K}(A)$.
\end{proof}

\section{Concluding remarks and open problems}

We end the paper with some remarks on the sharpness and limitations of
our results.

For sequences of exponential growth, Theorem \ref{thm:exp} appears to
be reasonably sharp. Indeed, the lower bound would tend to $1$ for $K$
tending to $b$, which at least for integral $b$ seems a natural
restriction on the digits. However, the proof requires us to assume
that $K^2 < n$ in order to ensure sufficiently large gaps in the
construction, so values of $K$ approaching $b$ are not
allowed. Removing this restriction is an open problem.

Concerning Theorems \ref{thm:double-exp2Nn} and
\ref{thm:double-expN2n}, we suspect that the bounds obtained here are
not best possible. We are not aware of work calculating upper bounds
on sets of the exact same form as those considered here, but for
related sets there are upper bounds by Han\v{c}l and \v{S}ustek
\cite{zbMATH05844058}. However, their sets assume a uniform bound on
the `digits' of any individual member of the set, which is not the
case here. Getting an upper bound on the Hausdorff dimension for the
sets of this paper remains an open problem.

\subsubsection*{Acknowledgements}

We thanks an anonymous referee and Demi Allen for their helpful
comments, which substantially improved the presentation of our
results.

\subsubsection*{Funding}

MG's and SK's research is supported by Aarhus University Research
Foundation, grant no. AUFF-E-2021-9-2. JH is supported by the
University of Ostrava.

\subsubsection*{Data Availibility}

Data sharing is not applicable to this article as no new data were created or analyzed in this study.

\subsubsection*{Declaration}

The authors declare that they have no conflict of interest.


\begin{thebibliography}{10}

\bibitem{zbMATH03435566}
{\sc A.~V. Aho and N.~J.~A. Sloane}, {\em Some doubly exponential sequences},
  Fibonacci Q., 11 (1973), pp.~429--437.

\bibitem{zbMATH07126822}
{\sc S.~B. Andersen and S.~Kristensen}, {\em Arithmetic properties of series of
  reciprocals of algebraic integers}, Monatsh. Math., 190 (2019), pp.~641--656.

\bibitem{zbMATH03578925}
{\sc P.~Erd{\H{o}}s}, {\em Some problems and results on the irrationality of
  the sum of infinite series}.
\newblock J. {Math}. {Sci}. 10, 1-7 (1975)., 1975.

\bibitem{zbMATH01994007}
{\sc K.~Falconer}, {\em Fractal geometry. {Mathematical} foundations and
  applications}, Chichester: Wiley, 2nd ed.~ed., 2003.

\bibitem{zbMATH01154148}
{\sc J.~Han{\v{c}}l}, {\em Transcendental sequences}, Math. Slovaca, 46 (1996),
  pp.~177--179.

\bibitem{zbMATH02103621}
\leavevmode\vrule height 2pt depth -1.6pt width 23pt, {\em A criterion for
  linear independence of series}, Rocky Mt. J. Math., 34 (2004), pp.~173--186.

\bibitem{zbMATH06790024}
{\sc J.~Han{\v{c}}l and R.~Nair}, {\em On the irrationality of infinite series
  of reciprocals of square roots}, Rocky Mt. J. Math., 47 (2017),
  pp.~1525--1538.

\bibitem{zbMATH06090451}
{\sc J.~Han{\v{c}}l, R.~Nair, L.~Novotn{\'y}, and J.~{\v{S}}ustek}, {\em On the
  {Hausdorff} dimension of the expressible set of certain sequences}, Acta
  Arith., 155 (2012), pp.~85--90.

\bibitem{zbMATH05200846}
{\sc J.~Han{\v{c}}l, R.~Nair, and J.~{\v{S}}ustek}, {\em On the {Lebesgue}
  measure of the expressible set of certain sequences}, Indag. Math., New Ser.,
  17 (2006), pp.~567--581.

\bibitem{zbMATH05635538}
{\sc J.~Han{\v{c}}l, A.~Schinzel, and J.~{\v{S}}ustek}, {\em On expressible
  sets of geometric sequences}, Funct. Approximatio, Comment. Math., 39 (2008),
  pp.~71--95.

\bibitem{zbMATH05227390}
{\sc J.~Han{\v{c}}l and J.~{\v{S}}ustek}, {\em Expressible sets of sequences
  with {Hausdorff} dimension zero}, Monatsh. Math., 152 (2007), pp.~315--319.

\bibitem{zbMATH05844058}
\leavevmode\vrule height 2pt depth -1.6pt width 23pt, {\em Boundedly
  expressible sets.}, Czech. Math. J., 59 (2009), pp.~649--654.

\bibitem{zbMATH06714945}
\leavevmode\vrule height 2pt depth -1.6pt width 23pt, {\em Sequences of
  {Cantor} type and their expressibility}, Math. Slovaca, 67 (2017),
  pp.~41--50.

\bibitem{zbMATH07725151}
{\sc M.~L{\o}kkegaard~Laursen}, {\em Algebraic degree of series of reciprocal
  algebraic integers}, Rocky Mt. J. Math., 53 (2023), pp.~517--529.

\end{thebibliography}
\end{document}